\documentclass[11pt]{article}
\usepackage{amsmath,amssymb,amsthm,enumitem,graphicx,tikz}
\usepackage{mathrsfs}
\usepackage{amssymb,amsfonts,amsmath,color}

\newtheorem{theorem}{Theorem}
\newtheorem{lemma}{Lemma}
\newtheorem{problem}{Problem}

\title{Note on semi-proper orientations of outerplanar graphs\footnote{The first arXiv version of this paper was submitted in April 2020. The only technical differences between the first arXiv version and this one
are in terminology and notation. The results are the same and were obtained independently of \cite{DH} as we became aware of \cite{DH} only in July 2020 after \cite{DH} appeared online in Discrete Applied Mathematics.}}
\author{Ruijuan Gu$^1$, Gregory Gutin$^2$, \ Yongtang Shi$^3$, Zhenyu Taoqiu$^3$  \\[2mm]
{\small $^1$ Sino-European Institute of Aviation Engineering}\\
{\small Civil Aviation University of China, Tianjin 300300, China}\\
{\small millet90@163.com}\\
{\small $^2$ Department of Computer Science}\\
{\small Royal Holloway, University of London}\\
{\small Egham, Surrey, TW20 0EX, UK
}\\
{\small g.gutin@rhul.ac.uk}\\
{\small $^3$  Center for Combinatorics and LPMC}\\
{\small Nankai University, Tianjin 300071, China}\\
{\small shi@nankai.edu.cn, tochy@mail.nankai.edu.cn}
}
\date{\today}

\begin{document}
	\maketitle
\begin{abstract}
	A \textit{semi-proper orientation} of a given graph $G$, denoted by $(D,w)$, is an orientation $D$ with a weight function $w: A(D)\rightarrow \mathbb{Z}_+$, such that the in-weight of any adjacent vertices are distinct, where the \textit{in-weight} of $v$ in $D$, denoted by $w^-_D(v)$, is the sum of the weights of arcs towards $v$. The \textit{semi-proper orientation number} of a graph $G$, denoted by $\overrightarrow{\chi}_s(G)$, is the minimum of maximum in-weight of $v$ in $D$ over all semi-proper orientation $(D,w)$ of $G$. This parameter was first introduced by Dehghan (2019). When the weights of all edges eqaul to one, this parameter is equal to the \textit{proper orientation number} of $G$. The \textit{optimal semi-proper orientation} is a semi-proper orientation $(D,w)$ such that $\max_{v\in V(G)}w_D^-(v)=\overrightarrow{\chi}_s(G)$.
	
	Ara\'ujo et al. (2016) showed that  $\overrightarrow{\chi}(G)\le 7$ for every cactus $G$ and the bound is tight. We prove that for every cactus $G$, $\overrightarrow{\chi}_s(G) \le 3$ and the bound is tight.
	Ara\'{u}jo et al. (2015) asked whether there is a constant $c$ such that  $\overrightarrow{\chi}(G)\le c$ for all outerplanar graphs $G.$ While this problem remains open, we consider it in the weighted case. We prove that for every outerplanar graph $G,$ $\overrightarrow{\chi}_s(G)\le 4$ and the bound is tight.
	
	\noindent\textbf{Keywords:} proper orientation number; semi-proper orientation number; outerplanar graph
\end{abstract}

\section{Introduction}\label{intro}
\baselineskip 17pt
For basic notation in graph theory, the reader is referred to \cite{BM}. All graphs in this paper are considered to be simple.
An \textit{orientation} $D$ of a graph $G$ is a digraph obtained from $G$ by replacing each edge by excactly one of two possible arcs with the same endvertices. The \textit{in-degree} of $v$ in $D$, denoted by $d_D^-(v)$, is the number of arcs towards $v$ in $D$ for each $v\in V(G)$. We will use the notation without subscript when the orientation $D$ is clear from context.

For a given undirected graph $G$, an orientation $D$ of $G$ is \textit{proper} if $d^-(u)\ne d^-(v)$ for all $uv\in E(G)$. An orientation with maximum in-degree at most $k$ is called a \textit{$k$-orientation}. The \textit{proper orientation number} of a graph $G$ is the minimum integer $k$ such that $G$ admits a proper $k$-orientation, denoted by $\overrightarrow{\chi}(G)$.
The existence of proper orientation was demonstrated by Borowiecki et al. in \cite{BG}, where it was shown that every graph $G$ has a proper $\Delta(G)$-orientation, where $\Delta(G)$ is the maximum degree of $G$. Later, Ahadi and Dehghan \cite{AD} introduced the concept of the proper orientation number. This parameter was widely investigated recently, for more details, we refer the reader to \cite{AD,ADM,AG,AC,AH,KM}.
Note that every proper orientation of a graph $G$ induces a proper vertex coloring of $G$. Hence, we have the following sequences of inequalities:
\begin{align}\label{e1}
\omega(G)-1\le \chi(G)-1\le \overrightarrow{\chi}(G)\le \Delta(G)
\end{align}
These inequalities are best possible since, for a complete graph $K_n$, $\omega(K_n)-1=\chi(K_n)-1=\overrightarrow{\chi}(K_n)=\Delta(K_n)=n-1.$
Ahadi and Dehghan \cite{AD} proved that it
is NP-complete to compute $\overrightarrow{\chi}(G)$ even for planar graphs. Araujo et al. \cite{AC}
strengthened this result by showing that it holds for bipartite planar graphs of
maximum degree 5.
%The non-monotonicity of proper orientation number makes it difficult to prove upper bounds on the parameter even for relatively narrow classes of graphs. Recall that a graph parameter $\gamma$ is \textit{monotonic} if $\gamma(H)\le \gamma(G)$ for every (induced) subgraph $H$ of $G$.
The following two problems have received great attention by researchers.
\begin{problem}[\cite{AC}]\label{pro1}
Is there a constant $c$ such that $\overrightarrow{\chi}(G)\le c$ for every planar graph $G$?
\end{problem}
\begin{problem}[\cite{AH}]\label{pro2}
	Is there a constant $c$ such that $\overrightarrow{\chi}(G)\le c$ for every outerplanar graph $G$?
\end{problem}
\noindent Knox et al.~\cite{KM} proved that $\overrightarrow{\chi}(G)\le 5$ for a $3$-connected planar bipartite graph $G$ and Noguci~\cite{N} showed that $\overrightarrow{\chi}(G)\le 3$ for any bipartite planar graph $G$ with $\delta(G)\ge 3$.
Araujo et al. \cite{AH} proved $\overrightarrow{\chi}(G)\le 7$ for any cactus, i.e., an outerplanar graph with every 2-connected component being either an edge or a cycle and $\overrightarrow{\chi}(T)\le 4$ for any tree $T$ (see
also \cite{KM} for a short algorithmic proof).
Ai et al.~\cite{AG} proved that $\overrightarrow{\chi}(G)\le 3$ for any triangle-free, $2$-connected outerplanar graph $G$ and $\overrightarrow{\chi}(G)\le 4$ for any triangle-free, bridgless or tree-free outerplanar graph $G$. Later, Ara\'{u}jo et al. \cite{AS} studied the notion of a weighted proper orientation of graphs.

Recently Dehghan \cite{D} introduced the notion of a semi-proper orientation of graphs.
A \textit{semi-proper orientation} of a given graph $G$, denoted by $(D,w)$, is an orientation $D$ with a weight function $w: A(D)\rightarrow \mathbb{Z}_+$, such that the in-weight of any adjacent vertices are distinct, where the \textit{in-weight} of $v$ in $D$, denoted by $w^-_D(v)$, is the sum of the weights of arcs towards $v$. Let $\mu^-(D,w)$ be the maximum of $w^-_D(v)$ over all vertices $v$ of $G.$
We drop the subscript when the orientation and weight function are clear from the context. The \textit{semi-proper orientation number} of a graph $G$, denoted by $\overrightarrow{\chi}_s(G)$,
is the minimum of $\mu^-(D,w)$ over all semi-proper orientations $(D,w)$ of $G$. An \textit{optimal semi-proper orientation} is a semi-proper orientation $(D,w)$ such that $\mu^-(D,w)=
\overrightarrow{\chi}_s(G)$. %Dehghan \cite{D} also studied the \TQ{semi-proper orientation} number under the name ``\TQ{semi-proper orientation} number", and proved the following result.

\begin{theorem}[\cite{D}]\label{D}
	Every graph $G$ has an optimal semi-proper orientation $(D,w)$ such that the weight of each edge is one or two.
\end{theorem}

It is easy to see that $\overrightarrow{\chi}_s(G)\le \overrightarrow{\chi}(G)$. Moreover, by the definition of a semi-proper orientation, the in-weights of adjacent vertices are different.
%When the weights of all edges eqaul to one, the orientation is just a \textit{proper orientation}, the in-weight of $v$ is just the in-degree of $v$, and $\TQ{\overrightarrow{\chi}_s(G)}$ is equal to $\overrightarrow{\chi}(G)$.
Consequently, by (\ref{e1}), we have
\begin{align} \label{e2}
\omega(G)-1\le \chi(G)-1\le \overrightarrow{\chi}_s(G)\le \overrightarrow{\chi}(G)\le \Delta(G)
\end{align}
Dehghan \cite{D} observed that there exist graphs $G$ such that $\overrightarrow{\chi}_s(G) < \overrightarrow{\chi}(G).$
Indeed, while as observed in \cite{D}, we have $\overrightarrow{\chi}_w(T)\le 2$ for every three $T$,
there are trees $T$ with $\overrightarrow{\chi}(T)=4$ \cite{AC}.
Thus, one natural problem is to study the gap between these two parameters.
\begin{problem}[\cite{D}]\label{pro3}
	Is there any constant $c_1$ such that $\overrightarrow{\chi}(G)-\overrightarrow{\chi}_s(G)\le c_1$ for every graph $G$?
\end{problem}
In this paper, we prove a sharp upper bound for the semi-proper orientation number of cacti in Theorem \ref{cacti}, which implies that $c_1\ge 4$ if $c_1$ exist, due to the sharp upper bound $\overrightarrow{\chi}(G)\leq 7$ for cacti proved in \cite{AH}.

%In \cite{AS}, the authors proved that it is (weakly) NP-complete to determine whether $\TQ{\overrightarrow{\chi}_s(G)}\leq k$ for trees, but can be solved by a pseudo-polynomial time algorithm. They also present a dynamic programming algorithm to determine whether a general graph $G$ on $n$ vertices and treewidth at most $tw$ satisfies $\TQ{\overrightarrow{\chi}_s(G)}\leq k$.
In \cite{D}, Dehghan showed that determining whether a given planar graph $G$ with $\overrightarrow{\chi}_s(G)=2 $ has an optimal semi-proper orientation $(D,w)$
such that the weight of each edge is one is NP-complete. He also proved that the
problem of determining the semi-proper orientation number of planar bipartite graphs is NP-hard.

%Ara\'{u}jo et al. (2016)  obtained a tight upper bound for $\overrightarrow{\chi}(G)$ on cacti and
%Ara\'{u}jo et al. \cite{AS} asked to obtain a tight upper bound for $\TQ{\overrightarrow{\chi}_s(G)}$ on cacti.
	 We prove the following two results. Theorem \ref{cacti} gives a  tight bound for cacti in the weighted case. Note that this theorem and the tight bound on the proper orientation number of cacti
	 imply that $c_1\ge 4$ in Problem \ref{pro3} (provided that $c_1$ exists). While Problem \ref{pro2} remains open, we consider the problem in the weighted case. Theorem \ref{main} solves this problem.
Due to Theorem~\ref{D}, the bounds in these theorems can be achieved for optimal semi-proper orientations where every edge weight is 1 or 2.

\begin{theorem} \label{cacti}
For every cactus $G,$ we have $\overrightarrow{\chi}_s(G) \le 3$ and this bound is tight.
%Moreover, every cactus $G$ has an optimal \TQ{semi-proper orientation} where the weight of each edge is one or two.
\end{theorem}

\begin{theorem}\label{main}
For every outerplanar graph $G,$ we have $\overrightarrow{\chi}_s(G) \le 4$ and this bound is tight.
%Moreover, every  $G$ has an optimal \TQ{semi-proper orientation} where the weight of each edge is one or two.
\end{theorem}
While the tightness proof  of the bound in Theorem \ref{cacti} is quite easy, that in Theorem \ref{main} is more involved as an optimal semi-proper orientation of a significantly larger graph is considered.
%Furthermore, for attacking Problems \ref{pro1} and \ref{pro4}, we study (weighted) proper orientations of some triangulated planar graphs with weight at most two or all equal to one.
%\begin{theorem}\label{tr}
%(i) For $n\ge 3$ and every iterated triangulation Tr($n$), we have $5\leq \overrightarrow{\chi}_w(\text{Tr}(n))\le 6.$ [{\color{green} LB is tight, but UB is unclear}]\\
%%Moreover,  every Tr($n$) has an optimal \TQ{semi-proper orientation} in which the weight of each edge is one or two.
%(ii) For every triangulated grid Gr($n$), we have $\overrightarrow{\chi}(\text{Gr}(n))\le 4$ and the bound is tight.\\
%%Moreover, every Gr($n$) has a proper orientation in which weight of each edge at most two. Both bounds are tight.\\
%(iii) For every fish $F_{x,y}$, we have $\overrightarrow{\chi}(F_{x,y})\le 4$ and the bound is tight.
%\end{theorem}
%\noindent From Theorem~\ref{tr}(i), we know that if $c_2$ in Problem~\ref{pro4} exists, then $c_2\ge6$. \cGG{Check it later.}

The remainder of the paper is organized as follows. We provide some definitions and simple lemmas for orientations on paths in Section~\ref{pre}. Next, we study (weighted) proper orientations of cacti and outerplanar graphs, and prove Theorems~\ref{cacti} and \ref{main} in Sections~\ref{cacti1} and \ref{outer}, respectively. %We conclude the paper in Section \ref{conc}.

\section{Preliminaries}\label{pre}
Let us consider briefly some graph theory terminology and notation used in this paper. For more information on blocks and ear decomposition, see e.g. \cite{BM}.

We denote a path and cycle by $P$ and $C,$ respectively, and the order of $P$ and $C$ by $|P|$ and $|C|,$ respectively. We call an edge $e$ an {\it $a$-$b$ edge} if the end points of $e$ have in-weight $a$ and $b$, respectively.

A \textit{block} of a graph $G$ is a maximal nonseparable subgraph of $G$ and a block of order $i$ is said to be an \textit{$i$-block}. Note that every $i$-block with $i\ge 3$ is a 2-connected graph, 2-block is an edge (bridge) of $G$ and 1-block is an isolated vertex of $G.$ Thus, if $G$ is connected, it has no 1-blocks.

%With the fact of blocks, we may use block tree to represent the block decomposition of a graph.
The \textit{block tree} associated to $G$ is the tree $T(G)$ with vertex set $V(T(G))=\{v_i\colon B_i\text{ is block of }G \}\cup S$, where $S$ is the set of cut vertices of $G$, and edge set $E(T(G))=\{v_is_j\colon s_j\in B_i \}$.
Choose a block $B_0$ of $G$ as a root of $T(G)$, and run depth-first search (DFS) algorithm on $T(G)$ from $B_0$. Then we can get an ordering of blocks in $G$ as $B_0,B_1,\dots,B_p$. If a cut vertex $s_i\in B_i\cap B_j$ and $j<i$, then say $s_i$ is the \textit{root} of $B_i$.

For a subgraph $H$ of $G$, an \textit{ear} of $H$ is a non-trivial path $P$ in $G$ with end-vertices in $H$ but internal vertices not. We say an ear is \textit{attached} to the corresponding ends in $H$ and we call such pair of end-vertices \textit{active}. Especially, if the vertices of an active pair are adjacent to each other, we call the pair \textit{active edge}. It is well known that every 2-connected graph $G$ has an {\em ear decomposition} defined as follows.
\begin{itemize}
	\item Choose a cycle $C_0$ of $G$ and let $G_0=C_0$.
	\item Add an ear $P_i$ attached to an active pair $(a_i,b_i)$ of $G_i$, where $a_i\ne b_i$ and let $G_{i+1}=G_i\cup P_i$, $0\le i<k$.
	\item $G_k=G$.
\end{itemize}

%Following \cite{AST} we define some special classes of planar graphs.

Now we introduce a class of outerplanar graphs, called \textit{universal outerplanar graphs}, which will be used in our proof.
A \textit{universal outerplanar graph}, denoted by UOP($n$), is defined as follows:
\begin{itemize}
	\item UOP(1) is a triangle.
	\item Add 2-length ears to all edges of UOP(1), then we get UOP(2). The new added edges are called \textit{outeredges} of UOP(2) and the new added vertices are called \textit{outervertices} of UOP(2).
	\item UOP($k+1$) is obtained from UOP($k$) by adding 2-length ears to all outeredges of UOP($k$).
\end{itemize}
%Observe that $\text{UOP($n$)}\subseteq \text{Tr($n+1$)}$.

We give some lemmas for orientations on paths below, which will be used later.

\begin{lemma}[\cite{AG}] \label{path1}
	Let $P=v_1v_2\dots v_n $ be a path of length $n-1$.
	\begin{enumerate}
		\item\label{p1-0} If $n\geq 7$, then there are three semi-proper orientations with weights of all edges one such that $w^{-}(v_1)=0$ and $w^{-}(v_n)=0$ and
		\begin{enumerate}
			\item $w^{-}(v_2)=2$, $w^{-}(v_{n-2})=0$, $w^{-}(v_{n-1})=2$, and
			\item $w^{-}(v_2)=1$, $w^{-}(v_3)=2$, $w^{-}(v_{n-3})=0$, $w^{-}(v_{n-2})=2$, $w^{-}(v_{n-1})=1$, and
			\item $w^{-}(v_2)=1$, $w^{-}(v_{n-2})=0$, $w^{-}(v_{n-1})=2$, respectively.
		\end{enumerate}
		\item\label{p1-1} If $n=6$, then there are two semi-proper orientations with weights of all edges one such that $w^{-}(v_1)=0$ and $w^{-}(v_6)=0$ and
		\begin{enumerate}
			\item $w^{-}(v_2)=2$, $w^{-}(v_3)=0$, $w^{-}(v_4)=1$, $w^{-}(v_5)=2$, and
			\item $w^{-}(v_2)=1$, $w^{-}(v_3)=2$, $w^{-}(v_4)=0$, $w^{-}(v_5)=2$, respectively.
		\end{enumerate}	
		\item\label{p1-2} If $n=5$, then there are two semi-proper orientations with weights of all edges one such that $w^{-}(v_1)=0$ and $w^{-}(v_5)=0$ and
		\begin{enumerate}
			\item $w^{-}(v_2)=1,$ $w^{-}(v_3)=2,$ $w^{-}(v_4)=1,$ and
			\item $w^{-}(v_2)=2,$ $w^{-}(v_3)=0,$ $w^{-}(v_4)=2,$ respectively.
		\end{enumerate}
	    \item\label{p1-3} If $n=4$, then there exists a semi-proper orientation with weights of all edges one such that $w^{-}(v_1)=0,$ $w^{-}(v_2)=1,$ $w^{-}(v_3)=2$ and $w^{-}(v_4)=0.$
	\end{enumerate}
\end{lemma}

\begin{lemma} \label{path2}
Let $P=v_1v_2v_3$ be a path with length two. Then there exist three semi-proper orientations with weights of all edges at most two such that $w^-(v_1)=w^-(v_3)=0$,
	\begin{enumerate}
		\item\label{p2-1} weights of all edges are one and $w^-(v_2)=2$.
		\item\label{p2-2} $w^-(v_2)=3$, $w(v_1v_2)=2$ and $w(v_3v_2)=1$.
		\item\label{p2-3} $w^-(v_2)=4$ and $w(v_1v_2)=w(v_3v_2)=2$.
	\end{enumerate}
\end{lemma}
\begin{proof}
	\begin{enumerate}
		\item \includegraphics[scale=0.2]{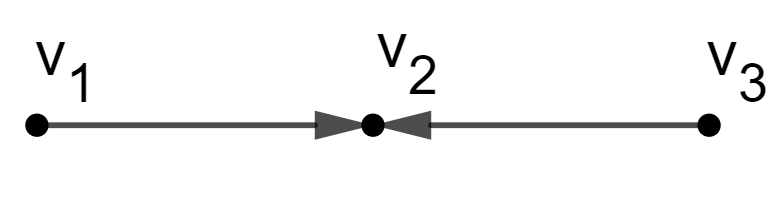}
		\item \includegraphics[scale=0.2]{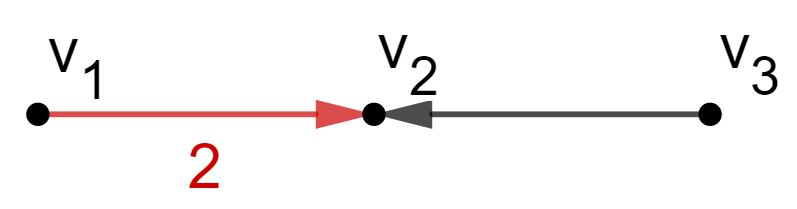}
		\item \includegraphics[scale=0.2]{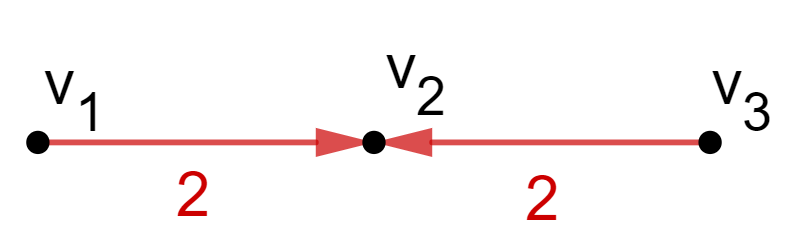}
	\end{enumerate}
\end{proof}

\begin{lemma} \label{path3}
	Let $P=v_1v_2v_3v_4$ be a path of length $3$. Then there exist two semi-proper orientations with weights of all edges at most two such that $w^-(v_1)=w^-(v_4)=0$,
	\begin{enumerate}
		\item\label{p3-1} $w^-(v_2)=2$, $w^-(v_3)=3$, $w(v_1v_2)=w(v_3v_4)=2$ and $w(v_2v_3)=1$.
		\item\label{p3-2} $w^-(v_2)=1$, $w^-(v_3)=3$, $w(v_2v_3)=2$ and $w(v_1v_2)=w(v_3v_4)=1$.
	\end{enumerate}
\end{lemma}
\begin{proof}
	\begin{enumerate}
		\item \includegraphics[scale=0.2]{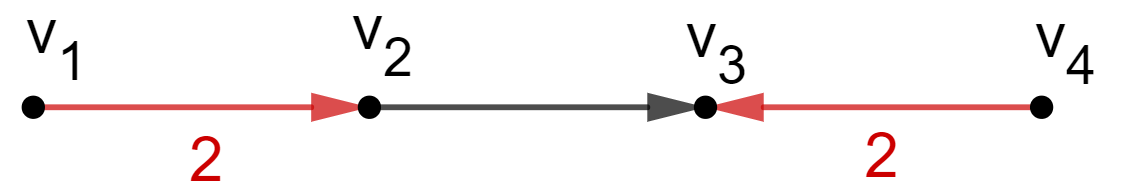}
		\item \includegraphics[scale=0.2]{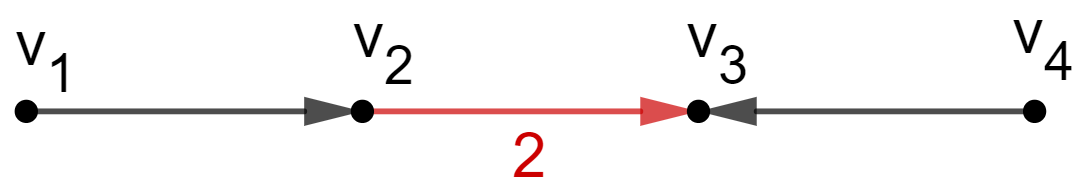}
	\end{enumerate}
\end{proof}

\begin{lemma}\label{path4}
	Let $P=v_1v_2v_3v_4v_5$ be a path with length $4$. Then there exists a semi-proper orientation with weights of all edges at most two such that $w^-(v_1)=w^-(v_3)=w^-(v_5)=0$, $w^-(v_2)=2$, $w^-(v_4)=3$, $w(v_4v_5)=2$ and other edges with weight one.
\end{lemma}
\begin{proof}
	\includegraphics[scale=0.2]{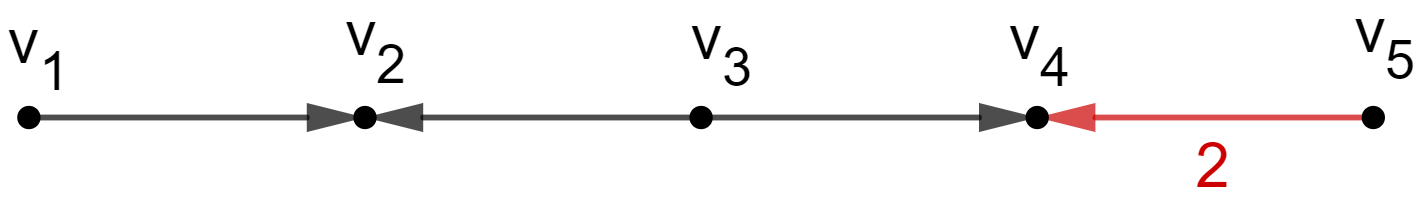}
\end{proof}

\begin{lemma}\label{path5}
	Let $P=v_1v_2v_3v_4v_5v_6$ be a path with length $5$. Then there exists a semi-proper orientation with weights of all edges at most two such that $w^-(v_1)=w^-(v_6)=0$, $w^-(v_2)=w^-(v_5)=1$, $w^-(v_3)=3$, $w^-(v_4)=2$, $w(v_2v_3)=w(v_4v_5)=2$ and  other edges with weight one.
\end{lemma}
\begin{proof}
	\includegraphics[scale=0.2]{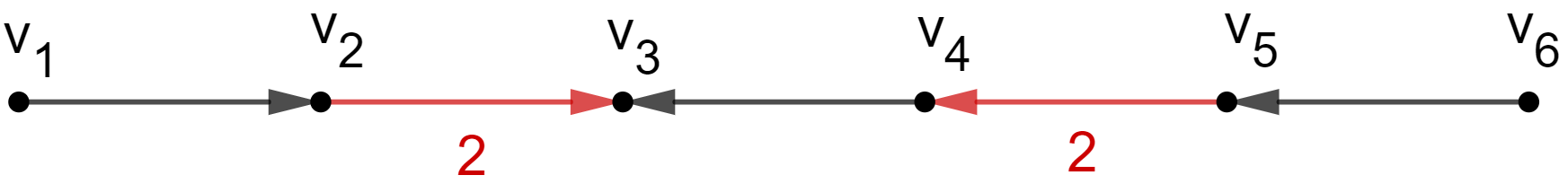}
\end{proof}

\section{Proof of Theorem~\ref{cacti}}\label{cacti1}
Since $G$ is a cactus, we can label blocks of $G$ and construct a sequence of induced subgraphs of $G$ as follows.
%contruct $G$ by the process of block decomposition as follows.
\begin{itemize}
	\item Choose a block $B_0$ as the root of block tree $T(G)$, i.e. $G_0=B_0$.
	\item Run DFS algorithm on $T(G)$ from $B_0$ to get an ordering of blocks in $G$.
	\item Add block $B_i$ to its root $s_i$, i.e., $G_i=G_{i-1}+B_i$, $1\le i\le k$.
	\item $G_k=G$.
\end{itemize}
Note that $B_i$ is either a $2$-block or a cycle in $G$.
	
We prove it by induction on $k$. When $k=0$, orient $B_0$ using Lemmas~\ref{path1} and  \ref{path2} or orient it arbitrarily if $B_0$ is 2-block. By the induction hypothesis, $G_{k-1}$ has a desired orientation $(D_{k-1},w)$.

Consider the case when $B_k$ is a cycle $C$. If $|C|=3$ we can apply Lemmas \ref{path3} and \ref{path1} by setting $v_1=v_4=s_k.$ If $w^-(s_k)=1$ in $(D_{k-1},w)$, then we use  Lemma~\ref{path3}-\ref{p3-1} to orient $C$ such that $w^-(v_2)=2$ and $w^-(v_3)=3$. If $w^-(s_k)=2$  in $(D_{k-1},w)$, then use Lemma~\ref{path3}-\ref{p3-2} to orient $C$ such that $w^-(v_2)=1$ and $w^-(v_3)=3$.
If $w^-(s_k)\in \{0,3\}$  in $(D_{k-1},w)$, then use Lemma~\ref{path1}-\ref{p1-3} to orient $C$ such that $w^-(v_2)=1$  and $w^-(v_3)=2$.

If $|C|=4,$ we can apply Lemma \ref{path1}-\ref{p1-2} by setting $v_1=v_5=s_k$ and using (a) if $w^-(s_k)\ne 1$ in $(D_{k-1},w)$ and (b) otherwise.
If  $|C|=5$ set $v_1=v_6=s_k$. If $w^-(s_k)=2$ in $G_{k-1}$, then use Lemma~\ref{path5}. Otherwise, use Lemma~\ref{path1}-\ref{p1-1}(a).
Now we assume $|C|\ge 6$ and set $v_1=v_n=s_k.$ If $w^-(s_k)=1$  in $(D_{k-1},w)$, then use Lemma~\ref{path1}-\ref{p1-0}(a) and otherwise Lemma~\ref{path1}-\ref{p1-0}(b).
	
Now consider the case when $B_k$ is a $2$-block $s_k v$. Then orient it from $s_k$ to $v$. If $w^-(s_k)=1$ in $(D_{k-1},w)$, then let $w(s_k v)=2$ such that $w^-(v)=2$. Otherwise, let $w(s_k v)=1$ such that $w^-(v)=1$.
	
For both cases, we have $w^-_{(D_k,w)}(s_k)=w^-_{(D_{k-1},w)}(s_k)$. This implies that $G_k$ has a  desired orientation $(D_k,w)$.

\begin{figure}[!htpb]
	\centering\includegraphics[scale=0.2]{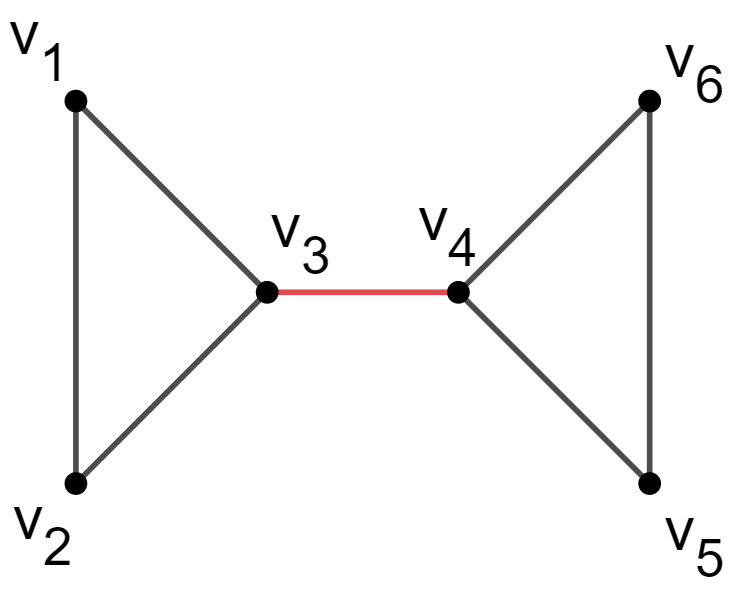}
	\caption{A tight example of Theorem~\ref{cacti}.}
	\label{gif}
\end{figure}

A tight example is given in Figure~\ref{gif}. Let $G$ be a cactus with vertex set $V(G)=\{v_1,v_2,v_3,v_4,v_5,v_6 \}$ and edge set $E(G)=\{v_1v_2,v_1v_3,v_2v_3,v_3v_4,v_4v_5,v_4v_6,v_5v_6 \}$. Suppose $G$ has a semi-proper 2-orientation $D$.
Without loss of generality, we may assume that edge $v_3v_4$ in $G$ is oriented in $D$ from $v_4$ to $v_3.$
Hence, $1\le w_D^-(v_3)\le 2.$ Since $w_D^-(v_3)\le 2,$ without loss of generality, we may assume that edge  $v_2v_3$ is oriented from $v_3$ to $v_2.$ Thus, $1\le w_D^-(v_2)\le 2.$
We cannot have $1\le w_D^-(v_1)\le 2$ as well since that would imply that there are vertices $v_i,v_j$ with $1\le i<j\le 3$ such that $w_D^-(v_i)=w_D^-(v_j).$ Hence, $w_D^-(v_1)=0,$ but then both $v_1v_2$ and $v_1v_3$ must be oriented from $v_1$ implying that $w_D^-(v_2)=w_D^-(v_3)=2,$  a contradiction.

\qed

\section{Proof of Theorem \ref{main}}\label{outer}

We start from the following:

%Follow from Theorem \ref{cacti}, we only consider graphs which are not cacti. We prove the following claim first.
\begin{lemma} \label{conn1}
	Let $G$ be a 2-connected outerplanar graph and let $s$ be an arbitrary vertex of $G.$ Then there exists a semi-proper orientation $(D,w)$ such that $\overrightarrow{\chi}_s(G) \le 4$
	and $w^-(s)=0$.
\end{lemma}
\begin{proof}
Since $G$ is $2$-connected, recall that we can construct $G$ by the process of ear decomposition as follows.
\begin{itemize}
	\item Choose a cycle $C_0$ containing $s$ and let $G_0=C_0$.
	\item Add an ear $P_i$ attached to an active pair $(a_i,b_i)$ of $G_i$, where $a_i\ne b_i$ and let $G_{i+1}=G_i\cup P_i$, $0\le i<k$.
	\item $G_k=G$.
\end{itemize}
Note that $a_i$ is adjacent to $b_i$ as $G$ is outerplanar.
	
We prove this lemma by induction on $k$. When $k=0$, orient $G_0$ using Lemma~\ref{path1} or orient it in any other proper way such that  $w^-(s)=0$.
By the induction hypothesis, $G_{k-1}$ has a desired orientation $(D_{k-1},w)$. Assume that $e=\{a_k,b_k\}$ is an active edge of $P_k=v_1v_2 \dots v_n$ and assume without loss of generality that $a_k=v_1$, $b_k=v_n$, $w^{-}(a_k)<w^{-}(b_k)$ in $(D_{k-1},w)$. We consider the following cases. %and note that in each case the in-weight of $a_k$ does not change implying that $w^-(s)=0$ in the weighted orientation $(D_{k},w)$ of $G_k.$

{\bf Case 1} $|P_k|=3$. If $e$ is a 2-3 edge in $(D_{k-1},w)$, then use Lemma~\ref{path2}-\ref{p2-3} to orient $P_k$ such that $w^-(v_2)=4$. If neither $a_k$ nor $b_k$ has in-weight 2, then use Lemma~\ref{path2}-\ref{p2-1} to orient $P_k$ such that $w^-(v_2)=2$. If neither $a_k$ nor $b_k$ has in-weight 3, then use Lemma~\ref{path2}-\ref{p2-2} to orient $P_k$ such that $w^-(v_2)=3$.

{\bf Case 2} $|P_k|=4$. If $w^-(a_k)=1$ or $w^-(b_k)=2$ in $(D_{k-1},w)$, then use Lemma~\ref{path1}-\ref{p1-3} to orient $P_k$ such that $w^-(v_2)=2$ and $w^-(v_3)=1$. If not, then use Lemma~\ref{path1}-\ref{p1-3} to orient $P_k$ such that $w^-(v_2)=1$ and $w^-(v_3)=2$.

{\bf Case 3}  $|P_k|=5$. If $e$ is a 1-2 edge in $(D_{k-1},w)$, then use Lemma~\ref{path4} to orient $P_k$ such that $w^-(v_2)=2$ and $w^-(v_4)=3$. If neither $a_k$ nor $b_k$ has in-weight 1, then use Lemma~\ref{path1}-\ref{p1-2}(a) to orient $P_k$ such that $w^-(v_2)=w^-(v_4)=1$. If neither $a_k$ nor $b_k$ has in-weight 2, then use Lemma~\ref{path1}-\ref{p1-2}(b) to orient $P_k$ such that $w^-(v_2)=w^-(v_4)=2$.

{\bf Case 4}  $|P_k|=6$. If $w^-(a_k)=1$ or $w^-(b_k)=2$ in $(D_{k-1},w)$, then use Lemma~\ref{path1}-\ref{p1-1}(b) to orient $P_k$ such that $w^-(v_2)=2$ and $w^-(v_5)=1$. If not, then use Lemma~\ref{path1}-\ref{p1-1}(b) to orient $P_k$ such that $w^-(v_2)=1$ and $w^-(v_5)=2$.

{\bf Case 5}  $|P_k|\ge 7$. If $w^-(a_k)=1$ or $w^-(b_k)=2$ in $(D_{k-1},w)$, then use Lemma~\ref{path1}-\ref{p1-0}(c) to orient $P_k$ such that $w^-(v_2)=2$ and $w^-(v_{n-1})=1$. If not, then use Lemma~\ref{path1}-\ref{p1-0}(c) to orient $P_k$ such that $w^-(v_2)=1$ and $w^-(v_{n-1})=2$.	
	
For all cases above, the in-weights of $a_k$ and $b_k$ in $(D_k,w)$ are the same as that in $(D_{k-1},w)$. This implies that $G_k$ has a desired orientation $(D_k,w),$ where in particular $w^-(s)=0$.
\end{proof}

To complete the proof of Theorem \ref{main}, it remains to consider the case when $G$ is connected but not 2-connected.
Let $B_0,B_1,\dots ,B_k$ be a list of blocks of $G$ such that for every $i\in \{0,1,2,\dots , k\},$ the subgraph $G_i$ of $G$ induced by the union of blocks  $B_0,B_1,\dots ,B_i$ is connected.
Such a list can be obtained e.g. by using DFS on $T(G)$ as described in the beginning of the previous section. Let $s$ be the root of  $B_k.$
We  prove the following extension of the theorem by induction on $i\in \{0,1,\dots ,k\}$:

{\em For every  $i\in \{0,1,\dots ,k\}$, $G_i$ has a semi-proper orientation $(D_i,w)$ such that $\mu^-(D_i,w)\le 4$ and if $s\in V(G_i)$ then $w^-(s)=0$.}

\vspace{1mm}

If $B_0$ is a 2-connected outerplanar graph, then by Lemma~\ref{conn1}, $G_0$ has a semi-proper orientation $(D_0,w)$ such that $\mu^-(D_0,w)\le 4$ and $w^-(s)=0$ if $s\in V(G_0)$.
If $B_0$ is an edge, then we orient the edge from $s$ to ensure that $w^-(s)=0$ if $s\in V(G_0)$ and arbitrarily, otherwise.
By the induction hypothesis, let $G_{i-1}$ have a desired orientation $(D_{i-1},w)$ such that $w^-(s)=0$ if $s\in V(G_{i-1}).$

First consider the case when $B_i$ is a 2-connected outerplanar graph. By Lemma~\ref{conn1}, $B_i$ has a semi-proper orientation $(D',w)$ such that $\mu^-(D',w)\le 4$ and $w^-(s)=0$
if $s\in V(B_i).$ Thus,  $(D',w)$ does not add the in-weight of $s$ and $w^-(s)=0$ in the resulting semi-proper orientation of $G_i$ provided $s\in V(G_i).$
If $B_i$ is an edge $e$ then orient it from $s$ if $s$ is an end-vertex of $e$ and arbitrarily, otherwise. Then we obtain a desired orientation as above.

Now we show the tightness of the bound. We will have $G$=UOP(4) as a tight example, which is depicted in Figure~\ref{f3}. Suppose $\overrightarrow{\chi}_s(G) \le 3.$
Since $G$ contains a $K_3$-subgraph, $\overrightarrow{\chi}_s(G) = 3.$ Let $D$ be an optimal semi-proper orientation of $G$ and let $V_i$ be the set of vertices in $D$
with in-weight $i\in \{0,1,2,3\}.$

Note that the vertices of $G$ can be partitioned into three size-8 sets $A,B,C$ such that every $K_3$-subgraph of $G$ has one vertex $a\in A$, $b\in B$ and $c\in C$ as depicted in Figure~\ref{f3}.
(In other words, $A,B,C$ is a proper 3-coloring of $G.$)
Let $S=\Sigma_{v\in V(G)}w_D^-(v).$ We have $S\ge \Sigma_{v\in V(G)}d_G(v)=45.$ For every $K_3$-subgraph of $G$ with vertices $a,b,c$ we have
$$\{w^-(a),w^-(b),w^-(c)\}\in \{\{1,2,3\}, \{0,2,3\}, \{0,1,3\}, \{0,1,2\}\}. $$
Thus, $S\le 8(1+2+3)=48$ implying that the gap between the upper bound and lower bound of $S$ is 3. Hence, $G$ has at most three edges of weight 2 in $D.$
By the lower bound, $|V_3|\ge 7.$

Suppose  $|V_3|=8.$ By propagation of in-weights from one $K_3$-subgraph to another $K_3$-subgraph sharing an edge with the former,
we will get four outervertices of in-weight 3 implying that $G$ has four edges of weight 2 in $D,$ a contradiction.
Thus, $|V_3|=7.$ If either $|V_1|<8$ or $|V_2|<8$ then $S<45$ implying that $|V_1|=|V_2|=8$ and $S=45.$ Hence,
all edges of $G$ have weight 1 in $D.$ However, by the propagation of in-weights, we will conclude that at least three
outervertices have in-weight 3 implying that $G$ has three edges of weight 2 in $D,$ a contradiction.
%Suppose $|V_0|=0$. Then by propagation of in-weights from one $K_3$ to another $K_3$ sharing an edge with the former,
%we can conclude that all vertices of $A$ ($B$ and $C$, respectively) have the same in-weight in $D$. Thus,
%four outervertices have in-weight 3 implying that $G$ has four edges of weight 2 in $D,$ a contradiction. Therefore, $|V_0|>0$ and $S\le 47.$
\qed

\begin{figure}[!htpb]
	\centering\includegraphics[scale=0.2]{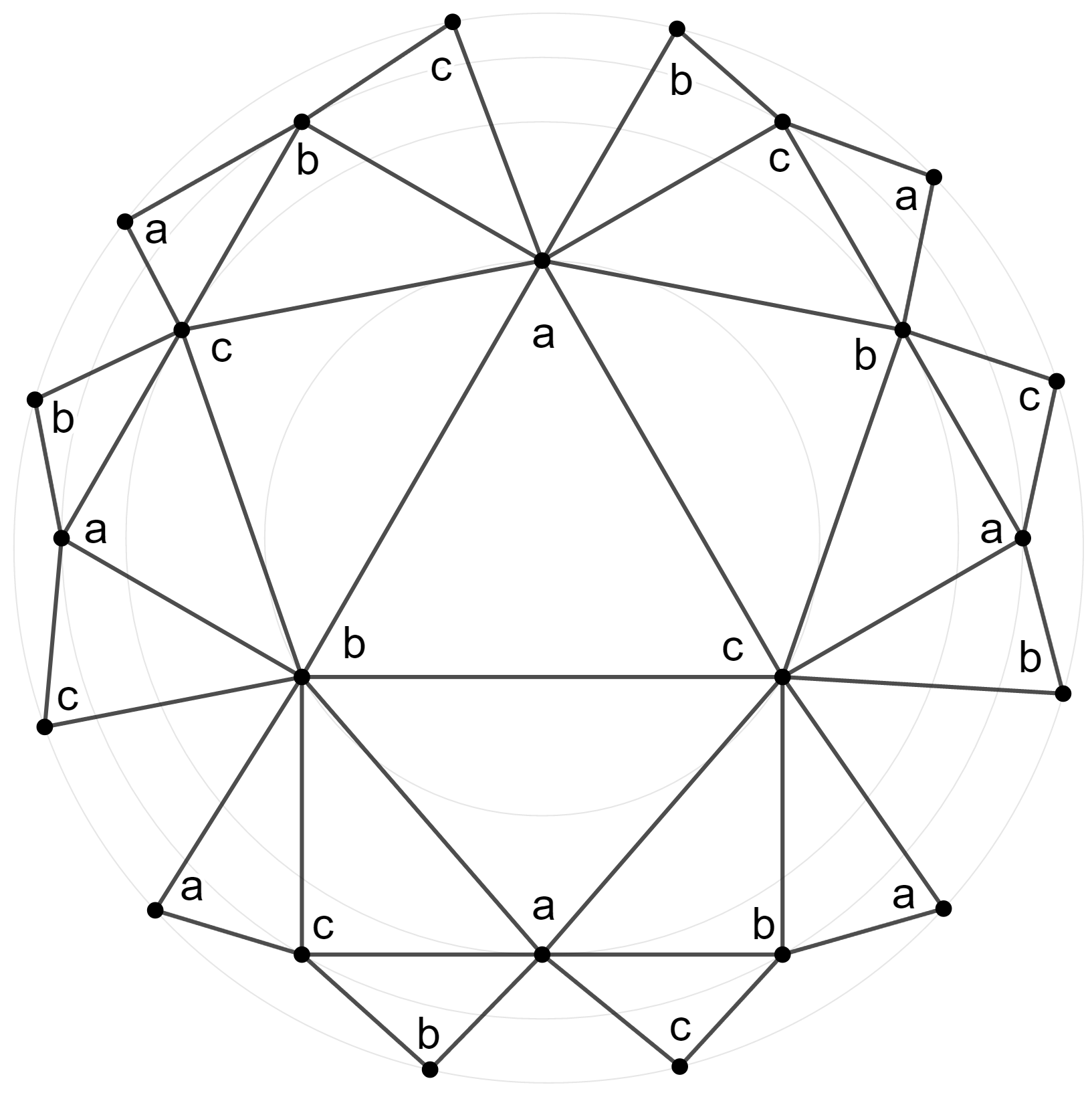}
	\caption{Optimal semi-proper orientation $D$ of UOP(4).}
	\label{f3}	
\end{figure}

%\section{Conclusion}\label{conc}
%Many authors have studied the proper orientation number in the last half dozen years. Perhaps, the main reason for difficulty of obtaining results
%for this parameter lies in it being not monotonic unlike many other graph parameters. Hopefully, approaches developed for the parameter can be
%used for other non-monotonic graph parameters.
%The main result of this paper solves the weighted version of Problem \ref{pro2}. If the answer is shown to be positive for the question of Problem \ref{pro3} on outerplanar graphs, then
%our result solves Problem \ref{pro2}, too. The next natural research direction is to attack the weighted versions of other appropriate problems stated in \cite{AC,AH}, in particular that of
%Problem \ref{pro1}.

\paragraph{Acknowledgments.}Shi and Taoqiu are supported by the National Natural Science Foundation of China (No. 11922112).

\end{document}